\documentclass[11pt]{article}

\usepackage{graphics}
\usepackage{graphicx}

\usepackage[a4paper, left=35mm,right=35mm,top=34mm,bottom=34mm]{geometry}
\usepackage[utf8]{inputenc}
\usepackage[T1]{fontenc}
\usepackage[english]{babel}

\usepackage{enumerate}
\usepackage{graphicx}
\usepackage{hyperref}
\hypersetup{
    colorlinks=true,
    linkcolor=blue,
    filecolor=magenta,      
    urlcolor=cyan,
}
\usepackage{lipsum}
\usepackage{amsfonts}
\usepackage{graphicx}
\usepackage{epstopdf}
\usepackage{algorithmic}

 \usepackage{dsfont}
 \usepackage{psfrag}
 \usepackage{color}
 \usepackage{tikz}
 \usepackage{subfigure}
\usepackage{mathtools,amsthm,amssymb,amsfonts}
\usepackage{algorithm}
\usepackage{algorithmic}
\makeatother
\theoremstyle{plain}
\usepackage{graphicx}        
\usepackage{multicol}        
\usepackage[bottom]{footmisc}
\def\bord{\partial\Omega}
\def\dTTr{\operatorname{\dot{T}r}}
\def\TTr{\operatorname{Tr}}
\def\Tr{\operatorname{tr}}

\def\loc{\operatorname{loc}}
\def\B{\mathcal{B}}
\def\Hp{\B(\bord)}
\def\Hm{\B'(\bord)}

\def\Nscr{\mathscr{N}}
\def\Ascr{\mathscr{A}}

\def\del{\partial}
\def\eps{\varepsilon}
\def\R{\mathbb{R}}

\def\dx{{\rm d}x}

\def\dmu{{\rm d}\mu}

\def\cpct{\operatorname{Cap}}
\def\supp{\operatorname{supp}}

\newcommand{\ddny}[2]{\frac{\partial_{#2} #1}{\partial\nu}}

\usepackage{newtxtext}       %
\usepackage[varvw]{newtxmath}       

\newtheorem{proposition}{\textbf{Proposition}}
\newtheorem{corollary}{\textbf{Corollary}}
\newtheorem{remark}{\textbf{Remark}}
\newtheorem{theorem}{\textbf{Theorem}}
\newtheorem{lemma}{\textbf{Lemma}}
\newtheorem{definition}{\textbf{Definition}}

\author{
  {\normalsize Anna Rozanova-Pierrat}\thanks{CentraleSup\'elec, Universit\'e Paris-Saclay, France
    (correspondence, anna.rozanova-pierrat@centralesupelec.fr).}
		}
	\title{On analysis of problems of mathematical physics with non-Lipschitz boundaries}	
\begin{document}
\maketitle

\abstract{We review recent advances in solving problems of mathematical physics on domains with irregular boundaries in $\R^n$. We distinguish two frameworks: a measure-free approach in the image of the trace operator spaces for extension domains and an $L^2$-approach depending on a $d$-upper regular boundary measure. In both cases, the domains can have boundaries with different Hausdorff dimensions inside the interval $(n-2,n)$. The generalization of the Poincaré-Steklov/Dirichlet-to-Neumann operator for these two contexts is given. To illustrate the established convergence of spectral problems for elliptic operators with Robin boundary conditions, we give a numerical example of the stability of localized eigenfunctions, using results of M. Graffin.}


\section{Introduction}
\label{sec:1}
Partial differential equations (PDEs) model the problems of mathematical physics. Geometrical mathematical surfaces or sets idealistically model the geometrical forms of real objects, where the different physical processes happen. Examples of such abstract mathematical geometrical models are planes, disques, spheres, cubes, and fractals~\cite{MANDELBROT-1967,MANDELBROT-1983}. There are several examples of followed works~\cite{BARDOS-2016,CAPITANELLI-2010,CAPITANELLI-2011,EVEN-1999,FLECKINGER-1995,GYRYA-2011,HINZ-2018,HINZ-2020,LAPIDUS-1996,ROZANOVA-PIERRAT-2012,ROZANOVA-PIERRAT-2025,SAPOVAL-1993,VAN_DEN_BERG-2000}.
 Fractal models naturally describe objects with a multi-scaling structure iterated up to infinity. This makes the fractal boundaries the most efficient in the heat exchanges~\cite{BARDOS-2016,DE_GENNES-1982,ROZANOVA-PIERRAT-2012,ROZANOVA-PIERRAT-2025}, optimal in the wave absorbtion~\cite{HINZ-2021-1,ROZANOVA-2020} and the most stable structures under loads~\cite{HINZ-2021}.

Historically, the theory of PDEs has developed on ``regular'' or ``sufficiently regular'' geometrical models of domain boundaries, which preserve useful properties of the corresponding solutions (see Section~\ref{sec:2} for the concepts and discussion). Typically, fractal boundaries were not initially considered and, up to now, are more understood as ``irregular'' shapes. 
This review starts with a discussion of what a ``regular'' and ``irregular'' boundary is. One of the properties that makes this difference natural is the existence of the normal vector and normal derivative on the boundary. We show that the notion of normal derivative, like many other classical notions, can be in a natural ``weak'' way generalized in a class of domains, containing the regular and irregular boundary shapes, without knowledge of the boundary measure. In this case, the functional spaces are based on the image of the trace operator, denoted in what follows by $\B(\del \Omega)$. In the framework of $\B(\del \Omega)$ and also its dual space $\B'(\del \Omega)$, all constructions on them are `` boundary measure free'' (see Sections~\ref{SubSecspaceB} and~\ref{SSPoincB}). In this framework, it is natural to work with Dirichlet or Neumann-type boundary conditions.

For Robin-type boundary conditions, it is more natural to consider an $L^2$-boundary framework where the boundary measure becomes crucial. It could also be viewed as a complementary approach, which is more natural for the regular boundaries. We show the generalization of the classical results on the regular boundary with the Hausdorff $(n-1)$-dimensional measure to the abstract boundary measures, which can be defined on multi-fractal structures. In this general framework, we define ``weakly regular'' shapes for which the trace operator is well-defined, which are called in different references as ``admissible domains'' with slight modifications in the definitions~\cite{ARFI-2019,CLARET-2025,DEKKERS-2022,HINZ-2021-1,HINZ-2023,ROZANOVA-PIERRAT-2021}. In this review, we do not give a particular name; each time, we describe the class of domains we are working on explicitly. 

Once the PDE problem in the irregular case is handled (its well-posedness is proven), there is a natural question about its approximation by the same type of problem on a regular domain. Theoretically, it is directly related to the convergence of a regular sequence of domains to an irregular one and the convergence of the corresponding energy and variational forms with the weak solutions (for instance, the Mosco convergence~\cite{MOSCO-1969} and also various publications as~\cite{CAPITANELLI-2010,CAPITANELLI-2011,DEKKERS-2022-1,DEKKERS-2022,HINZ-2023,LANCIA-2014}). Numerically, it is also related to discrete approximations of the boundary measures~\cite{Chandler-Wilde-2021,JOLY-2024}. The related question is shape optimization: Is it possible to find at least one optimal shape by realizing the minimum or maximum of a chosen energy criterion? For the convergence of solutions and other different operators constructed on them, it is crucial to consider ``a correct'' domain convergence, which could imply the desired approximation.

In this review, we introduce these different concepts compared to the classical regular case, present several recent results, and orient the reader to useful references. 

In Section~\ref{sec:2}, we introduce the trace operator for extension domains by developing first the ``measure-free'' framework in Subsection~\ref{SubSecspaceB} and then explaining the measure boundary framework in  Subsection~\ref{SSMeasureBF}. Based on these two contexts, we then introduce the normal derivative and the Dirichlet-to-Neumann/Poincaré-Steklov, studied in detail in Subsections~\ref{SSPoincB} and~\ref{SSPoincarL2} respectively, where the known results are developed for the general context of Section~\ref{sec:2}. In Section~\ref{SecApprox}, we discuss briefly the approximation of PDE problems on domains with irregular boundaries, focusing on Robin boundary value problems. In particular, we present a numerical illustration of the spectral stability~\cite[Theorem~6]{HINZ-2023} obtained in~\cite{GRAFFIN-2025}. Different standard definitions and notions for the reader's convenience are recalled in Section~\ref{Append}.

\section{From the trace theory to the solutions of elliptic boundary-value and transmission problems}
\label{sec:2}
Let $\Omega\subset \R^n$ be a connected open set, called \textit{a domain}, and let it be a bounded domain with a compact boundary $\del \Omega$. If $\del \Omega$ is Lipschitz~\cite[Definition~2.4.5]{HENROT-2005} or $C^1$ boundary~\cite[p.~626]{EVANS-2010} then there exists a normal vector $\nu$ to the boundary a.e. or everywhere, correspodingly, for the $(n-1)$-dimensional Lebesgue measure $\lambda^{n-1}$ on $\del \Omega$ (which is finite). Typically, $\nu(x)$ for $x\in \del \Omega$ is an element of $L^\infty(\del \Omega,\lambda^{n-1})$ often viewed as an element of $L^2(\del \Omega,\lambda^{n-1})$ or of more regular Sobolev space for sufficiently regular boundaries.
In this regular framework, the normal derivative on the boundary is understood as the projection (in $\R^n$) to the normal vector $\nu$ of the gradient of $u\in C^1(\del \Omega)$: $\frac{\del u}{\del \nu}= \nabla u\cdot \nu$. This quantity is an element of $L^2(\del \Omega, \mathcal{H}^{n-1})$, where $\mathcal{H}^{n-1}$ is the $(n-1)$-dimensional Hausdorff measure on the boundary, once $$u\in H^2(\Omega)=\{v\in L^2(\Omega)| \, \del_{x_i} v\in L^2(\Omega) \hbox{ and } \del^2_{x_i x_j}v \in L^2(\Omega) \hbox{ for all } i,j=1,\ldots,n\}.$$
The high interest to have $H^2$-regularity for the weak solutions of the elliptic boundary value problems is the access to various very confortable properties as an easy control of norms, working in the domain of the Laplacian and having for $n=2$ and a bounded $\Omega$ the Sobolev embedding $H^2(\Omega)\subset C(\overline{\Omega})$, which for instance, very useful for the well-posedness questions of non-linear models (for an example see~\cite{KALTENBACHER-2011-1} and the related references; for non-Lipschitz case see~\cite{DEKKERS-2022-1,DEKKERS-2022}). The main classical approaches and constructions fairly use this $H^2$-regularity and the existence of the normal vector on a boundary. 
Let us consider the example of the simplest elliptic problem for the Poisson equation with the homogeneous Dirichlet boundary condition
\begin{equation}~\label{EqPoisD}
	-\Delta u= f\in L^2(\Omega),\quad u|_{\del \Omega}=0
\end{equation}
on a bounded domain $\Omega$, and denote $H^1_0(\Omega)=\overline{C^\infty_0(\Omega)}^{H^1(\Omega)}$, defined as the closure in $H^1(\Omega)$ of the set of infinitely differentiable functions with compact support in $\Omega$, the Hilbert space with its usual norm $\|u\|_{H^1_0(\Omega)}=\|\nabla u\|_{[L^2(\Omega)]^n}.$ Problem~\ref{EqPoisD} is weakly well-posed for an arbitrary $\Omega$ by the Riesz' representation theorem: for all $f\in L^2(\Omega)$ there exists an unique solution $u\in H^1_0(\Omega)$ such that $(\nabla u,\nabla v)_{L^2(\Omega)}=(f,v)_{L^2(\Omega)}$ for all $v\in H^1_0(\Omega)$.
It is classicaly known~\cite[Theorem~2 p.~304 and Theorem~3 p.~316]{EVANS-2010} that $C^1$ boundary regularity provides the $H^2(\Omega)$ regularity of the weak solutions of~\eqref{EqPoisD}. However, it is no longer the case for a Lipschitz boundary with incoming corners (a non-convex domain $\Omega$~\cite{GRISVARD-1985}). Moreover, let us take the source term $f\in C^\infty_0(\Omega)$, then, from one hand, $u\in C^\infty(\Omega)\cap H^1_0(\Omega)$ for an arbitrary $\Omega$, but, from other hand, the $H^2$-property still depends on the regularity of the boundary~\cite{EDMUNDS-1987, EVANS-2010}, as it is called ``global'' or ``up to the boundary'' regularity of the solution. In~\cite{NYSTROM-1994} it is shown that for von Koch's snowflake $\Omega$ the weak solution $u\in H^1_0(\Omega)\cap C^{\infty}(\Omega)$ of~\eqref{EqPoisD} for all $f\in \mathcal{D}(\Omega)$  non negative and non identically zero is continuous up to the boundary $u\in C(\overline{\Omega})$, but $u\notin H^2(\Omega).$  

Suppose the boundary is non-Lipschitz with the Hausdorff dimension $d\in (n-1,n)$. In that case, the notion of the normal vector to the boundary does not exist, and all main definitions, as normal and tangential derivatives, must be redefined in a weak sense. To consider the problems of mathematical physics on domains with fractal and, more generally, non-Lipschitz boundaries, it is crucial first to define the trace operator on such an irregular boundary. We explain all notions for $H^1(\Omega)$ for simplicity. For more general spaces,  see, for instance~\cite{BIEGERT-2009,HINZ-2023,JONSSON-2009,ROZANOVA-PIERRAT-2021}. 
\subsection{Image of the trace or ``measure free'' framework}\label{SubSecspaceB}
As usual, we write $H^1(\Omega)$ for the Hilbert space of all $u\in L^2(\Omega)=L^2(\Omega,\R)$ such that $\nabla u \in L^2(\Omega,\R^n)$ and having the scalar product 
\begin{equation}\label{E:sp}
\left\langle u,v\right\rangle_{H^1(\Omega)}=\int_\Omega \nabla u\nabla v~\dx+\int_\Omega uv~\dx;
\end{equation}
here $\nabla u$ is interpreted in the distributional sense. The Hilbert space norms induced by \eqref{E:sp} is denoted by $\|\cdot\|_{H^1(\Omega)}$.

Let us define the classe of $H^1$-extension domains~\cite{HAJLASZ-2008,JONES-1981}:
\begin{definition}[$H^1$-extension domains]
A domain $\Omega\subset\R^n$ is called an \textit{$H^1$-extension domain}, if there is a bounded linear extension operator $\mathrm{E}_{\Omega}:H^1(\Omega)\to H^1(\R^n)$. 
It means that for all $u\in H^1(\Omega)$, $Eu|_{\Omega}=u$ and there exists a constant $C>0$ depending only on $n$ and $\Omega$ such that $\|Eu\|_{H^1(\R^n)}\le C \|u\|_{H^1(\Omega)}$.
\end{definition}
By Calderon-Stein~\cite{CALDERON-1961,STEIN-1970} any Lipschitz domain is an example of $H^1$-extension domains. By Jones~\cite{JONES-1981}, any $(\eps,\delta)$-domain (this notion, defined in Section~\ref{Append}, Definition~\ref{DefUnifD}, contains not only Lipschitz domains, but also domains with von Koch fractal boundaries, or more generally, $d$-set boundaries with $d\in (n-1,n)$~\cite{WALLIN-1991}, for the definition see Section~\ref{Append}, Definition~\ref{defdset}) is also an $H^1$-extension domain. Thanks to~\cite{HAJLASZ-2008} it is known that the optimal class of $H^1$-extension domains of $\R^n$ is described by two conditions:
\begin{enumerate}
	\item $\Omega$ is an $n$-set, or satisfy the density measure condition~\cite{HAJLASZ-2008}:
	 $$\exists c>0\quad \forall x\in \overline{\Omega}, \; \forall r\in (0,1] \quad \lambda(B_r(x)\cap \Omega)\ge C\lambda(B_r(x))=cr^n,$$
 where $\lambda(B_r(x))$ denotes the Lebesgue measure of the ball $B_r(x)=\{y\in \R^n|\, |x-y|<r\}$ in $\R^n$.  Therefore, an $n$-set
$\Omega$ cannot be ``thin'' close to its boundary $\del \Omega$, since it must always contain a non-trivial ball in its neighborhood.
	\item $H^1(\Omega)=C_2^1(\Omega)$ in the sense of equivalent norms, where $C_2^1(\Omega)$ is the space of the fractional sharp maximal functions~\cite{HAJLASZ-2008,ROZANOVA-PIERRAT-2021}.
	\end{enumerate}
	\begin{remark}
The density measure condition thus excludes domains with outgoing cusps and fractal tree structures, as in~\cite{ACHDOU-2013}, from our consideration. At the same time, the Hausdorff boundary dimension of an $H^1$-extension domain can vary in $[n-1,n)$. 
\end{remark}

Classical trace operator construction on a regular boundary is the following~\cite{EVANS-2010}: let $u\in H^1(\Omega)\cap C(\overline{\Omega})$, then its trace on $\del \Omega$ it is the values of its restriction on $\del \Omega$, $u|_{\del \Omega}$. Then, as $C(\overline{\Omega})$ is dense in $H^1(\Omega)$ and $C(\del \Omega)$ is dense in $L^2(\del \Omega,\mathcal{H}^{n-1})$, we result in a bounded continuous operator $\gamma: H^1(\Omega)\to   L^2(\del \Omega,\mathcal{H}^{n-1})$. This type of construction (see~subsection~\ref{SSMeasureBF}) is still valid for  $H^1$-extension domains once a boundary measure is fixed. 

Here, we explain how to introduce the trace operator without specifying the type of boundary measure. The construction follows~\cite{BIEGERT-2009} and holds in Lipschitz and non-Lipschitz cases~\cite{CLARET-2025,HINZ-2023}. 

We use the notions of capacity with respect to the space $H^1(\R^n)$ (see Definition~\ref{DefCap} \cite[Section 2.3]{CHEN-FUKUSHIMA-2012}, \cite[Section 2.1]{FOT94}  or \cite{ADAMS-1996,MAZ'JA-1985}). 
As it is known, a set of zero capacity has zero Lebesgue measure (we denote it by $\lambda^n$). 
A property which holds outside a set of zero capacity is said to hold \emph{quasi everywhere} (\emph{q.e.}).
Let $f\in L^1_{loc}(\R^n)$, the limit
\begin{equation*}\label{E:redefinition}
\widetilde{f}(x)=\lim_{r\to 0}\frac{1}{\lambda^n(\B(x,r))}\int_{\B(x,r)}f(y)dy
\end{equation*}
exists at $\lambda^n$-a.e. $x\in \mathbb{R}^n$. 
If $f\in H^1(\R^n)$, then it exists at $H^{1}(\mathbb{R}^n)$-quasi every $x\in\mathbb{R}^n$ and $\widetilde{f}$ defines a $H^{1}(\mathbb{R}^n)$-quasi-continuous version (representative) of $f$, \cite[Theorem 6.2.1]{ADAMS-1996}, \cite[Theorem 2.3.4]{CHEN-FUKUSHIMA-2012}, \cite[Theorem 2.1.3]{FOT94} or \cite{MAZ'JA-1985} (for the definition of a quasi-continuous function see~\ref{Append}).  Two quasi continuous representatives of the same element $f$ of $H^1(\mathbb{R}^n)$ agree q.e. on $\mathbb{R}^n$, see \cite[p. 71]{FOT94} or \cite[Theorem 6.1.4]{ADAMS-1996}. 
 \begin{definition}[Trace operator~\cite{CLARET-2025}]
 Let $\Omega$ be an $H^1$-extension domain with a boundary $\del \Omega$ of positive capacity.	By $\mathcal{B}(\partial\Omega)$ is denoted the vector space of all q.e. equivalence classes of pointwise restrictions $\widetilde{w}|_{\partial\Omega}$ of quasi-continuous representatives $\widetilde{w}$ of classes $w\in H^1(\R^n)$. Given $u\in H^1(\Omega)$, we choose an arbitrary element $w$ of $H^1(\mathbb{R}^n)$ such that $w=u$ a.e. in $\Omega$ and define the trace of $u$ on $\del \Omega$ by $\TTr u:=\widetilde{w}|_{\partial\Omega}$,  where $\widetilde{w}$ is an arbitrary quasi continuous representative $\widetilde{w}$ of $w$. 
 This defines
 the trace operator $\TTr u : \, H^1(\Omega) \to \mathcal{B}(\partial\Omega)$ by \begin{equation}\label{EqTraceDef}
                                                                                 	\TTr u=(\mathrm E_\Omega u)^\sim|_{\partial\Omega},
                                                                                 \end{equation}
 for any bounded linear extension operator $\mathrm{E}_\Omega:H^1(\Omega)\to H^1(\R^n)$.
 \end{definition}
As pointed in~\cite[Proposition~2.1]{CLARET-2025}, by \cite[Theorem 6.1 and Remark 6.2]{BIEGERT-2009} and  \cite[Corollary 6.3]{BIEGERT-2009}, the space $\mathcal{B}(\partial\Omega)$ defines the image of the trace operator,  $\TTr(H^1(\Omega))=\mathcal{B}(\partial\Omega)$:
\begin{proposition}\label{P:trace}
Let $\Omega$ be an $H^1$-extension domain with a boundary $\del \Omega$ of positive capacity. Then $u\mapsto \TTr u$ is a linear surjection $\TTr:H^1(\Omega)\to \mathcal{B}(\partial\Omega)$, well-defined in the sense that given $u\in H^1(\Omega)$, its trace $\TTr u$ on $\partial\Omega$ does not depend on the particular choice of $w$ or $\widetilde{w}$.  Moreover, definition \eqref{EqTraceDef} 
does not depend on the choice of a bounded linear extension operator $\mathrm E_\Omega: H^1(\Omega)\to H^1(\R^n)$.
\end{proposition} 
We also consider~\cite{CLARET-2025} the image of the trace operator $\mathcal{B}(\partial\Omega)$ as a Hilbert space:
\begin{proposition}
Let $\Omega\subset \R^n$ be an $H^1$-extension domain with a boundary $\del \Omega$ of positive capacity. Then, endowed with the quotient norm
\begin{equation}\label{EqTrNorm}
\left\|f\right\|_{\mathcal{B}(\partial\Omega)}:=\min\{ \left\|v\right\|_{H^1(\Omega)}\big|\ \text{$v\in H^1(\Omega)$ and $\dTTr\:v=f$}\},	
\end{equation}
the space $\mathcal{B}(\partial\Omega)$ is a Hilbert space. 
\end{proposition}
\begin{remark}
	Usually, such norms are defined by taking the infinum (cf.~\cite{HINZ-2021,HINZ-2023}), but here we draw attention to the fact that the infinum is realized and hence it is a minimum. Indeed, it is the direct corollary of the weak well-posedness of the Dirichlet problem for $-\Delta+1$ for the boundary data $f\in \Hp$,\cite{CLARET-2025}:  
	\begin{equation}\label{DO}
\begin{cases}
-\Delta u+u&=0\quad \text{in $\Omega$},\\
\qquad u|_{\partial\Omega}&=f.
\end{cases}
\end{equation}
\end{remark}


Given $f\in \Hp$, an element $u$ of $H^1(\Omega)$ is called a \emph{weak solution} of the Dirichlet problem \eqref{DO}
if $\TTr_i u=f$ and $\left\langle u,v\right\rangle_{H^1(\Omega)}=0$ for all $v\in H^1_0(\Omega)$.
Let $V_1(\Omega)$ denote the orthogonal complement of $H^1_0(\Omega)$ in $H^1(\Omega)$,
\begin{equation}\label{V1}
H^1(\Omega)=H^1_0(\Omega)\oplus V_1(\Omega).
\end{equation}
For any $f\in \Hp$, the Dirichlet problem \eqref{DO} has a unique weak solution $u^f$; this is immediate from \eqref{V1} or also from Stampacchia's theorem. By the first line in \eqref{DO}, a solution $u^f$ is called \emph{$1$-harmonic} on $\Omega$. It is in $V_1(\Omega)$, which is the space of all elements of $H^1(\Omega)$ that are $1$-harmonic on $\Omega$. Therefore, it is natural to consider Dirichlet $1$-harmonique extension operator $E_D=(\dTTr|_{V_1(\Omega)})^{-1}:f\in \Hp \mapsto u^f\in V_1(\Omega)$, which is linear and surjective isometry. The restriction $\dTTr|_{V_1(\Omega)}: V_1(\Omega)\to  \Hp$ is the inverse of $E_D$, as also isometry onto~\cite{CLARET-2025}.
\begin{remark}
	If $\Omega$ is a bounded Lipschitz domain, then, up to equivalent norms, $\Hp$ equals $H^{1/2}(\partial\Omega)$,~\cite[Chapter IV, Appendix]{DUTRAY-LIONS1988}. The space $H^{1/2}(\partial\Omega)$ can be endowed with an explicit norm of fractional Sobolev type; this norm involves the surface measure $\sigma$ on $\partial\Omega$. Going to works~\cite{JONSSON-1994,JONSSON-1984}, if $\partial\Omega$ is the support of a measure $\mu$ satisfying certain scaling conditions (see also Section~\ref{SSMeasureBF}), then the trace space $\Hp$ is a Besov type space. It can be endowed with an explicit norm involving $\mu$, see \cite{JONSSON-1994,JONSSON-1984}. Here we do not specify any measure on $\partial\Omega$ and do not impose the condition to be Lipschitz. Moreover, $\partial\Omega$ may have parts of different Hausdorff dimensions. The present formulation is ``measure-free'' and does not give any explicit norm representation for $\Hp$ using the minimal assumptions. 
\end{remark}
 
We summarize the trace operator properties in the following theorem (see also~\cite[Theorem~2.9]{CLARET-2025}): 

\begin{theorem}\label{Trisom}
Let $\Omega\subset \R^n$ be an $H^1$-extension domain with $\del \Omega$ of positive capacity. Then there hold the following properties of the trace operator $\TTr$ on $\del \Omega$:
\begin{enumerate}
\item[(i)] The  kernel of $\TTr$ is equal to $H^1_0(\Omega)=\ker\TTr=\{v\in H^1(\Omega)|\, \TTr v=0\}$.
\item[(ii)]  The image $\operatorname{Im}\TTr= \Hp$ is a Hilbert space endowed with norm~\eqref{EqTrNorm}. 
\item[(iii)] The trace operator $\TTr: H^1(\Omega)\to\Hp$ is a linear bounded with the operator norm $\|\TTr\|_{\mathcal{L}(H^1(\Omega),\Hp)}=1$.  Its restriction to the space of $1$-harmonic functions $\dTTr|_{V_1(\Omega)}:V_1(\Omega)\to \mathcal{B}(\partial\Omega)$  is an isometry and onto ($i.e.$ a bijection).
\end{enumerate}
\end{theorem}

\begin{remark}
	Generally, this trace construction allows us to define the linear bounded trace operator $\mathrm{Tr}_\Gamma: H^1(\R^n)\to \mathcal{B}(\Gamma)$ for all closed $\Gamma\subset \R^n$~\cite{HINZ-2021-1,HINZ-2023}. It corresponds to the well-posedness of the following Dirichlet boundary value problem for $f\in \mathcal{B}(\Gamma)$:
	\begin{equation}\label{Eq:Dir-1H-Prob}
\mbox{find } u\in H^1(\R^n) \mbox{ such that } (\Delta-1) u = 0 \mbox{ in } \R^n\setminus \Gamma \mbox{ and } \TTr_\Gamma u = f.
\end{equation}
With definition $\mathrm{tr}_\Gamma:=\mathrm{Tr}_\Gamma|_{V_1(\R^n\setminus \Gamma)}$, the operator  $\mathrm{tr}_\Gamma: V_1(\R^n\setminus \Gamma)\to \mathcal{B}(\Gamma)$ is a unitary isomorphism. Here, we have also used the orthogonal decomposition $H^1(\R^n)=H^1_0(\R^n\setminus \Gamma)\oplus V_1(\R^n\setminus \Gamma)$, taking in mind that $H^1_0(\R^n\setminus \Gamma)=\ker\TTr_\Gamma$.  Let $E_\Gamma=
(\mathrm{tr}_\Gamma)^{-1}: \mathcal{B}(\Gamma) \to V_1(\R^n\setminus \Gamma)$. Thus, the unique weak solution of~\eqref{Eq:Dir-1H-Prob} is $u= E_\Gamma f$. 
\end{remark}
%

%

 
\subsection{Measure boundary framework}\label{SSMeasureBF}
 We follow~\cite{HINZ-2023} for a construction of the trace operator on a more general space than previously introduced for $H^1(\Omega)$. 
 
Let $\Omega\subset \mathbb{R}^n$ be an arbitrary domain. This time, we define 
\[H^{1,2}(\Omega)=\{f\in \mathcal{D}'(\Omega): \text{$f=g|_\Omega$ for some $g\in H^{1}(\mathbb{R}^n)$}\},\] 
where $\mathcal{D}'(\Omega)=\mathcal{L}(C_0^\infty(\Omega),\R)$ denotes the space of distributions on $\Omega$. The space $H^{1,2}(\Omega)$ is a Hilbert space \cite[Section 2.2]{Triebel2002} endowed with the norm  
\[\left\|f\right\|_{H^{1,2}(\Omega)}:=\inf \{ \left\|g\right\|_{H^{1}(\mathbb{R}^n)}:\ g\in H^{1}(\mathbb{R}^n), f=g|_\Omega\ \text{in $\mathcal{D}'(\Omega)$}\}.\] 

Suppose that $\mu$ is a nonzero finite Borel measure whose support $\operatorname{supp} \mu =\del \Omega$, satisfaying a $d$-upper regular condition
for a fixed $d\in (n-2,n)$:  there exists $c_d>0$ such that
\begin{equation}\label{Eqmu}
\forall x\in\bord,\; \forall r\in]0,1],\quad\mu(B_r(x))\le c_d\,r^d.
\end{equation}
 Given $f\in H^{1,2}(\Omega)$, we set 
\[\mathrm{Tr}f:=\tilde{g}|_{\del \Omega},\]
where $g$ is an element of $H^{1}(\mathbb{R}^n)$ such that $f=g|_\Omega$ in $\mathcal{D}'(\Omega)$ (trivially $\R^n$ is an extension domain!). By the arguments used to prove \cite[Theorem 6.1]{BIEGERT-2009} one can see that $\mathrm{Tr}$ is a well-defined linear map
$\mathrm{Tr}:H^{1,2}(\Omega)\to L^2(\del \Omega,\mu)$ in the sense that if $g,h\in H^{1}(\mathbb{R}^n)$ satisfy $g=h$ $\lambda^n$-a.e. in $\Omega$, then $\mathrm{Tr} g=\mathrm{Tr} h$ $\mu$-a.e. on $\del \Omega$. 
The trace $\mathrm{Tr}:H^{1,2}(\Omega)\to L^2(\del \Omega,\mu)$ is a bounded linear operator, and the image $\mathrm{Tr}(H^{1,2}(\Omega))$ agrees with $\mathrm{Tr}(H^{1}(\mathbb{R}^n))$. In particular, the continuity of $\mathrm{Tr}$ implies that for every measurable $E\subset \del \Omega$, 
 $$\mathrm{cap}(E)=0 \Rightarrow \mu(E)=0.$$

 Let $\Omega$ is an $H^1$-extension domain and hence $H^{1}(\Omega)=H^{1,2}(\Omega)$. To study properties of the trace operator $\mathrm{Tr}:H^{1}(\Omega)\to L^2(\del \Omega,\mu)$, we always can study the properties of $\mathrm{Tr}_{\del \Omega}:H^{1}(\R^n)\to L^2(\del \Omega,\mu)$, as soon as (for all $u\in H^1(\Omega)$ it holds formula~\eqref{EqTraceDef}) 
 $\mathrm{Tr}=\mathrm{Tr}_{\del \Omega}\circ E_\Omega$ with $E_\Omega: H^1(\Omega)\to H^1(\R^n)$ a linear bounded extension operator.
 
Thanks to~\cite[Proposition 2.1]{HINZ-2023}, we know that the trace space $\mathrm{Tr}_{\del \Omega} (H^{1}(\mathbb{R}^n))$ (endowed as previously with the quotient norm) does not depend on the choice of an individual measure $\mu$ but only on the choice of an equivalence class of measures having ${\del \Omega}$ as a quasi support, see also \cite[Section 4.6, p. 168]{FOT94}.

Therefore, we have the compactness of the trace operator $\mathrm{Tr}_{\del \Omega}$ in the $L^2$-framework   if ${\del \Omega}$ is a compact set of $\R^n$ (by ~\cite[Theorem 2.1,Corollary 2.2]{HINZ-2023}):
\begin{theorem}\label{ThL2denseTr}
	Let  $d\in (n-2,n]$ and $\mu$ be a finite Borel measure with support ${\del \Omega}=\supp\mu$ and satisfying \eqref{Eqmu}. Then $\operatorname{Tr}_{\del \Omega}: H^{1}(\mathbb{R}^n)\to L^2({\del \Omega},\mu)$ is a bounded linear operator  with norm depending only on $n$, a nonnegative power of $\mu({\del \Omega})$,   $d$ and $c_d$. This operator is compact for compact sets ${\del \Omega}$.
The trace image space $\mathrm{Tr}_{\del \Omega}(H^{1}(\mathbb{R}^n))=\B(\del \Omega)$ is dense in $L^2({\del \Omega},\mu)$.
\end{theorem}
The density of the trace image space in $L^2({\del \Omega},\mu)$ follows from the fact that the measure $\mu$ on ${\del \Omega}$ is Borel regular, which implies (see~\cite[Proposition~3]{ROZANOVA-PIERRAT-2021}) that $\{ v|_{\del \Omega}: \, v\in C_0^\infty(\mathbb{R}^n)\}$, which is dense in $C({\del \Omega})$ by the Stone-Weierstrass theorem for the uniform norm, is also dense in $L^2({\del \Omega},\mu)$~\cite[Theorem~2.11]{EVANS-2010}. We also note that $C_0^\infty(\mathbb{R}^n)$ is dense in $H^1(\R^n)$. 
Once the domain $\Omega$ is an $H^1$-extension domain, we have the density of $C^\infty(\overline{\Omega})$ in $H^1(\Omega)$. Therefore, by~\eqref{EqTraceDef}, we find the direct corollary~\cite{HINZ-2021,HINZ-2021-1,HINZ-2023} of Theorem~\ref{ThL2denseTr}:
\begin{theorem}\label{ThTrComp}
	Let $\Omega$ be an $H^1$-extension domain of $\R^n$  and $\mu$ be a finite Borel measure with a compact support ${\del \Omega}=\supp\mu$ satisfying \eqref{Eqmu} for $d\in (n-2,n)$. Then thace operator $\operatorname{Tr}: H^{1}(\Omega)\to L^2({\del \Omega},\mu)$ is a linear compact operator with the dense image $\mathrm{Tr}(H^{1}(\Omega))=\B(\del \Omega)$ in $L^2({\del \Omega},\mu)$. If $\del \Omega$ is not compact, the trace operator is a linear, bounded operator.
\end{theorem}

 Let us consider the inclusion of boundary spaces
\begin{equation}\label{EqIncl}
\mathcal{B}(\del \Omega) \subset L^2(\del \Omega,\mu).	
\end{equation}
 If the continuity of this inclusion is obvious by Theorem~\ref{ThL2denseTr} and the capacity property of $H^1$ (a zero capacity set is also $\mu$-negligeable set),
$$\|f\|_{L^2(\del \Omega,\mu)}\le C \|f\|_{\mathcal{B}(\del \Omega)}, \; C>0,$$
the question of injectivity of this inclusion, needed for speaking about ``embedding'' (a continuous injective linear mapping),  was recently considered in~\cite{CHANDLER-WILDE-2025,ARXIV-HINZ-2025}.
 In particular we have~\cite{CHANDLER-WILDE-2025} the following theorem
 \begin{theorem}\label{ThGTripleInj}
 	Let  $\mu$ be a Radon measure whose support is a compact set $\Gamma$ of $\R^n$, and $\mu$ is upper $d$-regular on $\Gamma$, for some $d\in (n-2,n]$, then $\mu$ and $\Gamma$ satisfy for some constant $C>0$,
\begin{equation} \label{eq:bd}
\|\mathrm{Tr}_\Gamma v\|_{L^2(\Gamma,\mu)} \leq C\|v\|_{H^1(\R^n)}, \qquad v\in H^1(\R^n).
\end{equation} 
In addition, the following are equivalent:
\begin{enumerate}
\item[(i)] For every $u\in H^1(\R^n)$,
 \begin{equation} \nonumber
 \tilde u = 0 \;\, \mu\mbox{-a.e. on } \Gamma\ \Leftrightarrow\ \tilde u=0 \mbox{ q.e. on } \Gamma.
 \end{equation}
\item[(ii)] inclusion $\mathcal{B}(\Gamma) \subset L^2(\Gamma,\mu)$ is injective and defines the embedding with dense range.
\end{enumerate}
 \end{theorem}
 The main idea for the proof of this theorem is to consider the mapping $i:\mathcal{B}(\Gamma) \to L^2(\Gamma,\mu)$, $i=\tilde{\mathrm{Tr}}_\Gamma\circ E_D$ with $\tilde{\mathrm{Tr}}_\Gamma: C_0^\infty (\R^n) \to L^2(\Gamma,\mu)$, $\tilde{\mathrm{Tr}}_\Gamma v:= v|_{\Gamma},$  for $v\in C_0^\infty (\R^n)$. Then statements (i) and (ii) are corollaries of~\eqref{eq:bd}. We find that $i:\mathcal{B}(\Gamma) \to L^2(\Gamma,\mu)$ is an embedding and $i(\mathcal{B}(\Gamma))$ is dense
 in $L^2(\Gamma,\mu)$, as stated in Theorem~\ref{ThTrComp}. Moreover, considering the joint operator of $i$, $i^*: L^2(\Gamma,\mu)\to \mathcal{B}'(\Gamma)$ with the usual notations for the dual spaces $\mathcal{B}'(\Gamma)=\mathcal{L}(\mathcal{B}(\Gamma),\R)$ and $(L^2(\Gamma,\mu))'=L^2(\Gamma,\mu)$, we obtain that $L^2(\Gamma,\mu)$ is densely embedded in $\mathcal{B}'(\Gamma)$, and we have the usual Gelfand triple: $ \mathcal{B}(\Gamma)\subset L^2(\Gamma,\mu) \subset \mathcal{B}'(\Gamma)$.

\section{Poincaré-Steklov operator}
The previous trace constructions are helpful for the generalization of various notions, such as the normal derivative and the Poincaré-Steklov (or Dirichlet-to-Neumann)  operator. Each time we could think about two different frameworks: the image of the trace ``measure free'' case, when we can establish the isometry onto properties and the $L^2(\del \Omega,\mu)$ case depending on the classe of boundary measures with the dense embeddings $ \mathcal{B}(\Gamma)\subset L^2(\Gamma,\mu) \subset \mathcal{B}'(\Gamma)$.

For simplicity, we mainly focus on the $\Delta-1$ operator, weak formulations of which are associated directly with the $H^1$-norm. For the case of $\Delta$, the analysis is analogous~\cite {CLARET-2025,ARXIV-CLARET-2024-1}, but we need to pay attention to work in more appropriate spaces where $\|\nabla u\|_{(L^2(\Omega))^n}$ becomes a norm.

\subsection{Image of the trace or ``measure free'' framework}\label{SSPoincB}
In Section~\ref{SSMeasureBF}, we have defined the solution of the Dirichlet type boundary problem for the operator $\Delta-1$.
Let us now consider, as in~\cite[Section~2.3]{CLARET-2025}, the following Neumann problem:
\begin{equation}\label{SO}
\begin{cases}
-\Delta u+u &= 0\quad \text{in $\Omega$},\\
\qquad\frac{\partial u}{\partial\nu}|_{\bord} &= g.
\end{cases}
\end{equation}
Suppose $\Omega\subset \R^n$ is an $H^1$-extension domain with a boundary of positive capacity. Given $g\in \Hm$, we call $u\in H^1(\Omega)$ a \emph{weak solution} of the Neumann problem 
if for all $v\in H^1(\Omega)$, we have
\begin{equation}\label{Neumansol}
\langle u,v\rangle_{H^1(\Omega)}=\langle g, \TTr v\rangle_{\Hm,\Hp}.
\end{equation} 
Eq.~\eqref{Neumansol} omits the symbol $\frac{\partial u}{\partial\nu}$ of the normal derivative, but it is implicitly viewed as an element of $\Hm$. By the Riesz representation theorem, for any $g\in \Hm$ the Neumann problem \eqref{SO} has a unique weak solution $u_g\in H^1(\Omega)$. However, it is also an element of $V_1(\Omega)$. We define $$\mathscr{N}:\Hm\to V_1(\Omega),\quad \mathscr{N}g:=u_g,$$ as  the linear operator mapping a given element $g$ of $\Hm$ to the unique weak solution  of \eqref{SO}. 
By the isometry onto properties of $\TTr|_{V_1(\Omega)}$ and of its adjoint operator, it holds~\cite{CLARET-2025}  $\|\mathscr{N}g\|_{H^1(\Omega)}= \|g\|_{\Hm}$, $g\in \Hm$, that is, $\mathscr{N}$ is an isometry.
In ~\cite[Corollary~2.19]{CLARET-2025}, we proved
\begin{proposition}
Let $\Omega$ be $H^1$-extension domain with a boundary of positive capacity and let $\Tr:=\TTr|_{V_1(\Omega)}$. The linear operator $\Tr\circ\,\Nscr: \Hm\to \Hp$ is an isometry and onto. It satisfies, for $g,h\in\Hm$, 
\[\left\langle g, \Tr\circ\,\Nscr h\right\rangle_{\Hm,\,\Hp}=\left\langle g, h\right\rangle_{\Hm} = \left\langle h, \Tr\circ\,\Nscr g\right\rangle_{\Hm,\,\Hp}.\] 
\end{proposition}
The notion of the abstract normal derivative on $H^1$-extension domains repeats the known duality construction $\left\langle \frac{\partial u}{\partial\nu}, \TTr v\right\rangle_{H^{-\frac12},\,H^{\frac12}}$ by the Green formula in the Lipschitz case with $\Hp=H^{\frac12}(\del \Omega)$ (see for instance \cite[Section VII.1, Lemma 1]{DUTRAY-LIONS1988}). Different variants of this definition
 have been defined and studied by various authors in different contexts, see for instance \cite[p. 218]{CONSTANTINESCU-CORNEA-1963}, \cite[Section 3.2]{LEJAN-1978} and \cite{LANCIA-2002}. Let us follow~\cite{CLARET-2025,HINZ-2021-1,HINZ-2023}.
 We define the space $H^1_\Delta(\Omega)$ as all elements $u\in H^1(\Omega)$ with $\Delta u\in L^2(\Omega)$, understanding $\Delta u$ in the distributional sense. Clearly, $V_1(\Omega)\subset H^1_\Delta(\Omega)$. Let us introduce $H_{\del \Omega}:=\TTr_\Omega\circ E_D:\, \Hp\to H^1(\Omega)$ a linear extension operator of norm one, uniquely defined (see~\cite[Corollary~3.1]{HINZ-2023}). For all $u\in H^1_\Delta(\Omega)$ we define  a bounded linear functional $\frac{\partial u}{\partial\nu} \in \Hm$ by the identity
		\begin{multline}\label{FracGreen}
				\left\langle \frac{\partial u}{\partial\nu}, 
		\psi \right\rangle_{\Hm,\Hp}\\
		=\int_\Omega H_{\del\Omega} \psi\: \Delta u\:dx + \int_\Omega \nabla H_{\del\Omega} \psi \: \nabla u  \:dx, \quad \psi \in \Hp.
		\end{multline}
We refer to the unique element $\frac{\partial u}{\partial\nu}$ of $\Hm$ as the \emph{normal derivative of $u$ on $\del \Omega$}. 

\begin{remark}\mbox{}
 The right hand side of equality (\ref{FracGreen}) remains unchanged if  $H_{\del\Omega} \psi$ is replaced by an arbitrary element $w$ of $H^1(\Omega)$ satisfying $\mathrm{Tr} w=\psi$ (see~\cite[Remark~4.2]{HINZ-2023} and \cite[Section 2.3]{FOT94}; \cite[Theorem 7]{HINZ-2021-1}). 
 In particular, the right-hand side of formula~\eqref{FracGreen} is exactly the Green formula:
 \begin{equation}\label{EqGreenInt}
\left\langle \frac{\partial u}{\partial\nu},\TTr v\right\rangle_{\Hm,\,\Hp}= \int_\Omega{(\Delta u)v\,\dx}+\int_\Omega \nabla u\nabla v\,\dx,\quad v\in H^1(\Omega),
\end{equation}
  known for the regular $\del \Omega$ to compare to outward normal vector (in the regular framework instead of  $\left\langle \frac{\partial u}{\partial\nu}, 
		\psi \right\rangle_{\Hm,\Hp}$ there is $\int_{\del \Omega} \nabla u\cdot \nu \psi d \sigma$ with the outward normal vector $\nu$). To define exterior normal derivatives of $u$ as for the regular case, it is sufficient to change ``$+$'' in the right-hand side of~\eqref{FracGreen} by ``$-$''. 
 \end{remark}
 The operator $\frac{\partial}{\partial\nu}: H^1_\Delta(\Omega)\to \Hm$ is linear and bounded in the sense that 
\begin{equation}\label{E:boundnormalder}
\left\|\frac{\partial u}{\partial\nu}\right\|_{\Hm}\leq \|u\|_{H^1(\Omega)}+\|\Delta u\|_{L^2(\Omega)}.
\end{equation}
If we restrict the normal derivative on the space of $1$-harmonic functions  $$\del_{\nu}:=\left.\frac{\partial}{\partial\nu}\right|_{V_1(\Omega)},$$ we obtain the isometry and onto properties~\cite[Corollary~3]{CLARET-2025}. 

\begin{corollary}\label{C:Di}
Let $\Omega$ be an $H^1$-extension domain with a boundary of positive capacity. Then the following assertions hold:
\begin{enumerate}
\item[(i)] Both $\Nscr:\Hm\to V_1(\Omega)$ and the operator 
$\del_{\nu}: V_1(\Omega)\to \Hm$ are isometries and onto, and  $\del_{\nu}=\Nscr^{-1}_1$.
\item[(ii)] For any $u,v\in V_1(\Omega)$ we have
\begin{equation*}\label{E:secondGreen}
\big\langle  \del_{\nu} u,\Tr_i v\big\rangle_{\Hm,\Hp}=\left\langle u,v\right\rangle_{H^1(\Omega)}=\big\langle  \del_{\nu} v,\Tr_i u\big\rangle_{\Hm,\Hp}.
\end{equation*}
\item[(iii)] The dual $\big(\del_{\nu}\big)^\ast:\Hp\to V_1'(\Omega)$ of $\del_{\nu}$ is an isometry and onto.
\end{enumerate}
\end{corollary}

Therefore, as the definitions of the trace and normal derivatives are known, it is natural to define the Poincaré-Steklov (or Dirichlet-to-Neumann) operator. Let us follow~\cite{ARFI-2019,CLARET-2025}. 

Let $\Omega\subset \R^n$ be a bounded $H^1$-extension domain. We write $\Delta_D$ for the self-adjoint Dirichlet Laplacian on $L^2(\Omega)$ and denote its spectrum by $\sigma(\Delta_D)$. Recall that $\sigma(\Delta_D)\subset (-\infty,0)$ is pure point with eigenvalues accumulating at minus infinity. 

Given $k\in\R$ and $f\in \Hp$, we call $u\in H^1(\Omega)$ a \emph{weak solution} of the Dirichlet problem 
\begin{equation}\label{EqDOk}
\begin{cases}
-\Delta u+ku&=0\quad \text{in $\Omega$}\\
\qquad u|_{\partial\Omega}&=f
\end{cases}
\end{equation}
if $\TTr u=f$ and $\left\langle u,v\right\rangle_{\dot{H}^1(\Omega)}+k\left\langle u,v\right\rangle_{L^2(\Omega)}=0$ for all $v\in H^1_0(\Omega)$. Problem \eqref{DO} is the special case for $k=1$ of problem~\eqref{EqDOk}.
 
If $k\in \R\backslash \sigma(\Delta_D)$, then for any $f\in \Hp$ there is a unique weak solution $u_f\in H^1_\Delta(\Omega)$ of \eqref{EqDOk}. For such $k$ one can define a linear operator $\Ascr_k: \Hp \to \Hm$ by 
\[\Ascr_k f:=\frac{\partial u_f}{\partial\nu}.\]
This operator is called the \emph{Poincar\'e-Steklov} (or \emph{Dirichlet-to-Neumann}) \emph{operator} associated with $(\Delta-k)$ on $\Omega$. See for instance~\cite{ARENDT-2012,ARFI-2019,FRIEDLANDER-1991,ROZANOVA-PIERRAT-2021} for studies of Poincar\'e-Steklov operators under more restrictive assumptions on~$\Omega$. For the $C^\infty$ case, it is the elliptic self-adjoint pseudo-differential operator of the first order $A_k: C^\infty(\del \Omega) \to C^\infty(\del \Omega)$~\cite[11-12 Chapter 7]{TAYLOR-1996}. For a bounded Lipschitz domain, the Dirichlet-to-Neumann operator $\Ascr_k: H^\frac{1}{2}(\del \Omega)\to H^{-\frac{1}{2}}(\del \Omega)$ is a linear continuous equal to its adjoint operator. In our general framework of $H^1$-extension domains we have the analoguous results~\cite{ARFI-2019,CLARET-2025}:
\begin{lemma}\label{L:PS}
Let $\Omega$ be a bounded $H^1$-extension domain. Then the following assertions hold:
\begin{enumerate}
\item[(i)] For any $k\in \R\backslash \sigma(\Delta_D)$, the Poincar\'e-Steklov operator $\Ascr_k:\Hp\to \Hm$ is a bounded linear operator and coincides with its adjoint. It is injective if and only if $k$ is not an eigenvalue of the self-adjoint Neumann Laplacian for $\Omega$.
\item[(ii)]   The Poincar\'e-Steklov operator $\Ascr_1:\Hp\to \Hm$ satisfies $\Ascr_1=\partial_{\nu} \circ(\Tr)^{-1}$. It is an isometry with inverse $\Ascr_1^{-1}:\Hm\to \Hp$ given by $\Ascr_1^{-1}=\Tr\circ \Nscr$.
\end{enumerate}
\end{lemma}
\begin{remark}
	 In the measure-free framework, associated as seen previously, with isometry and bijective properties of the trace operator, we can also generalize the operators of layer potentials~\cite{CLARET-2025}.
The main idea is to use all notions in the weak/variational sense. 
\end{remark}
\subsection{$L^2(\bord,\mu)$-framework}\label{SSPoincarL2}

If $u\in H^2(\Omega)$,~\cite[p. 2]{GIRAULT-1986}, then $\frac{\partial u}{\partial x_i}\in H^1(\Omega)$ and
\[\frac{\partial u}{\partial\nu}=\sum_{i=1}^n\TTr \left(\frac{\partial u}{\partial x_i}\right)\nu_i \in
L^2(\partial\Omega,\sigma),\] where $\sigma$ is the surface measure on $\partial\Omega$ and 
$\nu=(\nu_1,...,\nu_n)$ is the $\sigma$-a.e. defined outward unit normal on $\partial\Omega$,~\cite[Section I.1, Lemma 1.4]{GIRAULT-1986}. Therefore, it is natural to consider the Poincar\'e-Steklov operator $A_k: H^1(\del \Omega,\sigma)\to L^2(\del \Omega,\sigma)$ for $C^1$-boundaries. For the Lipschitz non-convex domains, we do not have anymore the $H^2$-regularity of the weak solutions of the ellipric boundary value problems, thus it is more natural to consider the Poincar\'e-Steklov $A_k: L^2(\partial\Omega,\sigma)\to  L^2(\partial\Omega,\sigma)$, also in the aim to work in the same Hilbert space. For the general class of boundaries, possibly non-Lipschitz, we work in the assumptions of Theorem~\ref{ThTrComp} for bounded domains. Thus, following the same ideas as in~\cite{ARENDT-2011,ARFI-2019} we also able to define $A_{k}$ as an unbounded operator with a compact resolvent on $L^2(\del \Omega,\mu)$ in the case $k\in \R\setminus \sigma(-\Delta_D)$.  

Using the Gelfand triple $\B(\bord)\subset L^2(\del \Omega,\mu) \subset \B'(\del \Omega)$, the consideration of the Poincaré-Steklov operator as operator from $L^2(\del \Omega,\mu)$ to $L^2(\del \Omega,\mu)$ leads to the additional difficulty: to ensure $\frac{\del}{\del \nu}u\in L^2(\del \Omega,\mu)$.  
Thus, we need to update the definition of the interior normal derivative in the sense of $L^2$.
 
\begin{definition}\label{DefL2NormDeri}\textbf{($L^2$-normal derivative)}
Let $\Omega$ be an $H^1$-extension domain of $\R^n$  and $\mu$ be a finite Borel measure with the compact support ${\del \Omega}=\supp\mu$ satisfying \eqref{Eqmu} for $d\in (n-2,n)$. Let in addition $u \in H^1_\Delta(\Omega)$. 
If there exists $\psi \in L^2(\del \Omega,\mu)$ such that for all $v \in H^1(\Omega)$ it holds:
\begin{equation}
\label{EqDerNF}
\int_{\Omega}{(\Delta u)v\,\dx}+ \int_{\Omega}{\nabla u \cdot \nabla v\,\dx} = \int_{\del \Omega}{\psi \TTr v\,\dmu},
\end{equation}
then $\psi$ is called  a  $L^2$-normal derivative of $u$ on $\del \Omega$, denoted,  as previously, by   $\frac{\del u}{\del \nu}  = \psi$. 

\end{definition}

\begin{remark}
If $\frac{\del u}{\del \nu} \in L^2(\bord,\mu)$ 
exists for some $u\in H^1_\Delta(\Omega)$, 
then it is unique.
\end{remark}

To define the Dirichlet-to-Neumann operator on $L^2(\del \Omega,\mu)$ for the solutions of~\eqref{EqDOk} associated to $-\Delta+k$, 
we follow~\cite{ARENDT-2007}, and assume that $\Omega$ is a bounded $H^1$-extension domain. 

This time, the trace operator $\mathrm{Tr}$ is considered from  $H^1(\Omega)$ to $L^2(\del \Omega,\mu)$, which makes it compact. From \cite[Theorem 1]{HINZ-2021-1}, we know the compactness of the embedding $H^1(\Omega)\hookrightarrow L^2(\Omega)$ holds for all bounded $H^1$-extension domains.

In~\cite[p. 6]{ARENDT-2012}), for all $k\in \R$ and $\lambda\in \R$ such that $k$ and $\lambda+k$ are not eigenvalues of the Dirichlet Laplacian,   the space
\begin{equation}\label{EqSpaceHk}
	H_{k,\lambda}(\Omega)=\{u\in H^1(\Omega)\;|\; (\Delta-k)u=\lambda u \hbox{ in the sense of distributions}\}
\end{equation}
was introduced.
In the same way as explained in~\cite[Lemma~2.2]{ARENDT-2012} or~\cite[Lemma~3.2]{ARENDT-2007}, with $\operatorname{Ker} \TTr=H_0^1(\Omega)$, we conclude that:
\begin{equation}\label{EqDecompH1}
	H^1(\Omega)=H^1_0(\Omega)\oplus H_{k,\lambda}(\Omega),
\end{equation}
which directly implies that $V_1(\Omega)=H_{k,\lambda}(\Omega)$ for all $k\in \R$ and $\lambda\in \R$ such that $k$ and $\lambda+k$ are not eigenvalues of the Dirichlet Laplacian.
 Therefore, decomposition~\eqref{EqDecompH1} of $H^1(\Omega)$ and the boundary measure properties satisfying Theorem~\ref{ThGTripleInj} ensure that the trace operator $\Tr=\TTr|_{H_{k,\lambda}(\Omega)}: H_{k,\lambda}(\Omega) \to L^2(\del \Omega,\mu)$ is injective. In addition, since $H_{k,\lambda}(\Omega)$ is a closed subspace of $H^1(\Omega)$, the injection $H_{k,\lambda}(\Omega) \to L^2(\Omega)$ is compact as soon as $\Omega$ is a bounded $H^1$-extension domain ($i.e.$ for which $H^1(\Omega)\hookrightarrow L^2(\Omega)$,~\cite{ARFI-2019,HINZ-2021,HINZ-2021-1}).

Then, we introduce the Dirichlet-to-Neumann operator $A_{k}:L^2(\del \Omega,\mu) \to L^2(\del \Omega,\mu)$ as the operator associated with the following bilinear form:
\begin{equation}\label{EqbilFormDefAL2}
a(\Tr u, \Tr v):=	\int_{\Omega} \nabla u\cdot\nabla v\,\dx+k\!\int_\Omega u v\,\dx, \quad \hbox{ for } u,v \in H_{k,0}(\Omega). 
\end{equation}
The definition comes from the following relation, resulting from Green's formula:
$$a(\Tr u, \Tr v)=\left\langle \del_\nu u, \Tr v\right\rangle_{\B'\!,\,\B}=\left\langle\frac{\del u}{\del \nu}, \Tr v\right\rangle_{L^2(\del \Omega,\mu)}.$$
By its definition, $a(\cdot,\cdot)$ is symmetric and continuous on $H_{k,0}(\Omega)$ ($i.e.$ on $V_1(\Omega)$).
In addition, by~\cite[Proposition~3.3]{ARENDT-2007} the form $a(\cdot,\cdot)$ is elliptic in the following sense: 
there exists $\eta(k)\ge0$ such that for all $u\in H_{k,0}(\Omega)$,
\begin{equation}\label{EqPositiva}
a(\Tr u, \Tr u)+\eta \|\!\Tr u\|^2_{L^2(\del \Omega,\mu)}\ge \frac{1}{2} \|u\|^2_{H^1(\Omega)}.
\end{equation}
That property of $a$ is based on the compactness of the embedding $H_{k,0}(\Omega)\hookrightarrow L^2(\Omega)$  
and on the injectivity of the trace from $H_{k,0}(\Omega)$ to $L^2(\del \Omega,\mu)$.

Therefore, we prove the following theorem.
\begin{theorem}\label{ThPSL2clas} Let $\Omega$ be a bounded $H^1$-extension domain of $\R^n$  and $\mu$ be a finite Borel measure with the compact support ${\del \Omega}=\supp\mu$ satisfying \eqref{Eqmu} for $d\in (n-2,n)$.
Let $k\in \R_+$ be outside the spectrum of $-\Delta$ with Dirichlet boundary conditions.
Then, the  Dirichlet-to-Neumann  operator $A_{k}$ for problem~\eqref{EqDOk}, considered  from $L^2(\del \Omega,\mu)$ to $L^2(\del \Omega,\mu)$, is 
the operator 
$A_{k}f=\frac{\del u}{\del \nu}$ in $L^2(\del \Omega,\mu)$ with 
\begin{multline*}
 dom(A_{k})=\big\{f\in L^2(\del \Omega,\mu)\;\big|\; \exists u\in H^1(\Omega) \hbox{ such that } \\\TTr u=f, \quad (-\Delta +k)u=0 \quad\hbox{and}\quad \exists \frac{\del u}{\del \nu} \in L^2(\del \Omega,\mu)\big\}
\end{multline*}
and is a self-adjoint positive operator with a compact resolvent. Therefore, $A_{k}$ has a discrete spectrum with eigenvalues
$$0=\lambda_0<\lambda_1\le \lambda_2\le \ldots, \quad \hbox{with } \lambda_j\to +\infty \; j\to +\infty$$
and the corresponding eigenfunctions form an orthonormal basis in $L^2(\del \Omega,\mu)$.
\end{theorem}
\begin{proof}

Let  $L$ be the operator associated to the bilinear continuous elliptic form $a(\cdot,\cdot)$ defined in~\eqref{EqbilFormDefAL2}. Let us show that $L$ is the Dirichlet-to-Neumann operator $A_{k}$. For $u\in H_{k,0}(\Omega)$ and $g\in L^2(\del \Omega,\mu)$ it holds $\Tr u\in L^2(\del \Omega,\mu)=:dom(L)$,  
therefore $L(\TTr u)= g$ if and only if for all $v\in H_{k,0}(\Omega)$ 
\begin{equation}\label{EqNormD}
\int_\Omega{\nabla u\cdot\nabla v\,\dx}+k\!\int_\Omega{uv\,\dx}=\int_{\del \Omega}{g \TTr v\ \mbox{d}\mu}. 
\end{equation}
First, let us suppose~\eqref{EqNormD}.
Then we notice that not only does it hold on $H_{k,0}(\Omega)$, but also on $H^1(\Omega)$, since for all $v\in H^1_0(\Omega)$ (as $-\Delta u+k u=0$ in the distributional sense):
\begin{equation}\label{EqNorm0}
\int_\Omega \nabla u\cdot\nabla v\,\dx+k\!\int_\Omega uv\,\dx=0=\int_{\del \Omega}{g \TTr v\,\dmu}. 
\end{equation}
By equation~\eqref{EqDOk} for $u$, we have
\begin{equation}
\int_\Omega \nabla u\cdot\nabla v\,\dx+\int_\Omega \Delta u v\,\dx=\int_{\del \Omega}{\psi \TTr v\,\dmu},
\end{equation}
which, by Definition~\ref{DefL2NormDeri}, means that for $u\in H^1(\Omega)$, thus with its trace $\TTr u\in L^2(\del \Omega,\mu)$ (it follows simply from $\TTr (H^1(\Omega))\subset L^2(\del \Omega,\mu)$), we have $\frac{\del u}{\del \nu}=g \in  L^2(\del \Omega,\mu)$, and finally, $L(\TTr u)= g$, ensuring that $\TTr u\in dom(L)$. This means that $L=A_{k}$.

Conversely, let us suppose that $f\in L^2(\del \Omega,\mu)$ and $A_k(f)=g\in L^2(\del \Omega,\mu)$. Then there exists $u\in H_{k,0}(\Omega)$ such that
\begin{equation*}
\Tr u=f,\quad\hbox{and}\quad \del_\nu u= \frac{\del u}{\del \nu}=g.
\end{equation*}
Then for all $v\in H^1(\Omega)$
\begin{align*}
\int_\Omega{\nabla u\cdot\nabla v\,\dx}+k\!\int_\Omega {uv\,\dx} &= \int_\Omega{\nabla u\cdot\nabla v\,\dx}+\int_\Omega{\Delta u v\,\dx}\\
&=\int_{\del \Omega}{\frac{\del u}{\del \nu} \TTr v\,\dmu}
=\int_{\del \Omega}{g \TTr v\,\dmu}.
\end{align*}
Thus it follows that $f\in dom(L)$ and $Lf=g$.

The compactness of the trace operator $\TTr: H^1(\Omega)\to L^2(\del \Omega,\mu)$ implies that $A$ has a compact resolvent. Thus, applying the Hilbert-Schmidt Theorem to complete the proof is sufficient.
\end{proof}

An element $f\in L^2(\bord,\mu)$ is in the domain $dom(A_{k})$ of the Poincaré-Steklov operator if and only if $(f, A_{k}f)\in [L^2(\bord,\mu)]^2$ is in the so-called Calderón subspace \cite{ARENDT-2012} defined as follows:
\begin{multline*}
\mathcal{C}_k=\left\{(f,g)\in [L^2(\del \Omega,\mu)]^2\;\left|\; \exists u\in H^1(\Omega) \text{  such that}\right.\right.\\\left.\left.(-\Delta +k)u=0,\;\TTr u|_{\del \Omega}=f \;\hbox{and}\; \left.\frac{\del u}{\del \nu}\right|_{\del \Omega}=g\right.\right\}.
\end{multline*}
For all $k\in \R$, the Calderón space is defined as the set of the boundary data in $L^2$ for which a weak solution exists to the Cauchy problem for $-\Delta+k$. As mentioned in~\cite{MAGOULES-2021,MAGOULES-2025}, by analogy to~\cite[Theorem~1.2,~p.12]{DARDE-2010}, we have the uniqueness (with no assumption of existence) of the solution to the Cauchy problem for all $k\in\R$.  
On the other hand, if we set $f$ and consider all $g=g(f)$ for which there exists a solution to the Dirichlet problem with normal derivative equal to $g$, the solution $u$ is unique if and only if $k\in\R\setminus \sigma(-\Delta_D)$. This is also connected to the fact that the Cauchy problem is not well-posed in the sense of Hadamard and, for example, the mapping $u\in V_1(\Omega)\mapsto (\Tr u|_{\del \Omega}, \del_\nu u|_{\del \Omega})\in \Hp\times \Hm$ is injective but not onto.

Following~\cite{ARENDT-2012}, the Calderón space, defined for all $k\in \R$, also allows us to define the Poincaré-Steklov operator for all $k\in \R$. 
If $k\notin  \sigma(-\Delta_D)$, then $\mathcal{C}_k\cap (\{0\}\times L^2(\del \Omega,\mu))=\{(0,0)\}$
and there exists a densely defined closed operator $A_{k}$ on $L^2(\del \Omega,\mu)$ the graph of which is $\mathcal{C}_k$ (see Theorem~\ref{ThPSL2clas}).

Now, let $k$ be an eigenvalue of the Dirichlet Laplacian.
Then, by the Hilbert-Schmidt theorem, a finite number of eigenfunctions corresponding to $k$ exists. We consider the closed space of their $L^2$-normal derivatives:
\begin{align*}
P(k)&=\{h\in L^2(\del \Omega,\mu)\;|\; (0,h)\in \mathcal{C}_k\}\\
&=\left.\left\{\left.\ddny{v}{i}\right|_{\del \Omega}\;\right|\; v\in \operatorname{Ker}(-\Delta_D+k),\; \left.\ddny{v}{i}\right|_{\del \Omega}\in L^2(\del \Omega,\mu)\right\}.
\end{align*}
Therefore, we have $L^2(\del \Omega,\mu)=P(k)\oplus P(k)^\perp$. Denoting by $L^2_k(\del \Omega,\mu)$ the orthogonal of $P(k)$, $L^2_k(\del \Omega,\mu)$ is an infinite-dimensional closed subspace of $L^2(\del \Omega,\mu)$.
Thus, we can define the Poincaré-Steklov operator $A_{k}$ on the space $L^2_k(\del \Omega,\mu)$, which becomes equal to $L^2(\del \Omega,\mu)$ for $k\notin \sigma(-\Delta_D)$. 
\begin{theorem}\label{ThPSL2}
Let $\Omega$ be a bounded $H^1$-extension domain of $\R^n$  and $\mu$ be a finite Borel measure with the compact support ${\del \Omega}=\supp\mu$ satisfying \eqref{Eqmu} for $d\in (n-2,n)$.
Then, for all $k\in \R$, the Calderón space is a closed subspace of $L^2(\del \Omega,\mu)\times L^2(\del \Omega,\mu)$ and the Poincaré-Steklov operator $A_{k}$ is the unique self-adjoint operator with a compact resolvent on $L^2_k(\del \Omega,\mu)$ with a graph given by $$\{(f,A_{k}f)\}=\mathcal{C}_k\cap (L^2_k(\del \Omega,\mu)\times L^2_k(\del \Omega,\mu)).$$
	\end{theorem}
The proof is linked to the analogous result to~\cite[Proposition 1]{GREGOIRE-1976}, which also holds in our case by~\cite[Theorem 7.1]{HINZ-2021} (see also~\cite[Theorem~2.1]{MAGOULES-2021}):
for all $h\in L^2(\del \Omega,\mu)$, $k\in \R$ and $s\in \R^*$, there exists a unique $u\in H^1(\Omega)$ such that:
\begin{equation}\label{EqVFComplex}
\forall v\in H^1(\Omega), \quad \int_\Omega \nabla u\cdot\nabla\overline{v}\,\dx+k\int_\Omega u\overline{v}\,\dx=
 \int_{\del \Omega} (is\TTr u -h) \TTr \overline{v}\, \dmu,
\end{equation}
corresponding to the weak well-posedness of the problem $(-\Delta +k)u=0$ in $\Omega$ with $is\Tr u-\del_\nu u=h$ on $\del \Omega$. The result can be easily obtained using the first Fredholm theorem using the injectivity of the Cauchy problem for $-\Delta+k$.
Then the operator $$R(is):h\in L^2(\del \Omega,\mu)\mapsto \TTr u\in L^2(\del \Omega,\mu)$$ is a linear compact operator. Indeed, it can be seen as the composition of the linear bounded operator $h\in L^2(\del \Omega,\mu)\mapsto u\in H^1(\Omega)$ with the solution $u$ of~\eqref{EqVFComplex} (by the closed graph theorem) and the compact trace operator. 
Therefore, as proved in~\cite[Proposition~2]{ARENDT-2012}, for $s\in \R^*$ and $h\in L^2_k(\del \Omega,\mu)$ the resolvent of $A_{k}$ is given by:
\begin{equation*}
 	(is -A_{k})^{-1} h= R(is) h.
 \end{equation*}
The rest of the proof of Theorem~\ref{ThPSL2} also follows without modifications from~\cite[Proposition~2]{ARENDT-2012}. In particular, the self-adjoint property of $A_{k}$ follows from the inversibility of $(is-A_{k})$ for $s\in \R^*$. 

For the definition of the Poincaré-Steklov operator in the $L^2$-framework on unbounded non-Lipschitz domains, see~\cite{ARFI-2019}. The results were presented for $d$-set boundaries, but hold without modifications for all $H^1$-extension domains $\Omega$ of $\R^n$ associated with a finite Borel measure $\mu$ of the compact support ${\del \Omega}=\supp\mu$ satisfying \eqref{Eqmu} for $d\in (n-2,n)$. The main ideas of the construction are given by~\cite{ARENDT-2012-1,ARENDT-2015}.

\section{Approximation: irregular by regular and irregular by irregular cases}\label{SecApprox}
The question about the approximation of a fractal boundary by a more regular, for instance, a Lipschitz boundary is the most relevant in comparing the differences between the corresponding solutions, spectral properties, and energies. Theoretical analysis of the convergence of domains implying the corresponding convergence of solutions, energies, and spectra can also be viewed as the first step to the numerical analysis.
We are mostly interested in these kind of questions for problems with a Robin-type boundary condition on  a part (of a positive capacity) or all boundary:
\begin{equation}\label{EqRC}
	\frac{\del u}{\del \nu}+\alpha \TTr u=0 \hbox{ on }\Gamma \subset \del \Omega,
\end{equation}
where $\alpha$ is a real or complex constant coefficient.
As the trace operator maps to  $L^2(\del \Omega,\mu)$, condition~\eqref{EqRC} makes the normal derivative more regular than just an element of $\B'(\bord)$: it ensures that $\frac{\del u}{\del \nu}\in L^2(\del \Omega,\mu)$.
However, from the convergence of domains point of view, the Robin-type condition makes it necessary to understand the passage to the limit of energies containing the terms of this kind $\|\TTr u_k\|_{L^2(\Gamma_k,\mu_k)}$.
In~\cite[Lemma~2.1]{HINZ-2023} we prove (see also \cite[Theorem 5]{HINZ-2021-1}) 
\begin{lemma}\label{L:traceconvergence}  
Let $D\subset\mathbb{R}^n$ be a bounded Lipschitz domain, $0\leq d\leq n$ and $(n-d)/p<\beta$. Let $(\mu_m)_m$ be a sequence of Borel measures with supports $\Gamma_m=\supp\mu_m$ contained in $\overline{D}$ and such that \eqref{Eqmu} holds for all $m$ with the same constant $c_d$. Suppose that $(\mu_m)_m$ converges weakly to a Borel measure $\mu$. Then $\Gamma:=\supp\mu$ is contained in $\overline{D}$, and if $(v_m)_m\subset H^{\beta,p}(D)$ converges to some $v$ weakly in $H^{\beta,p}(D)$, then 
	\[\lim_{m\to\infty}\int_{\Gamma_{m}} \mathrm{Tr}_{D,\Gamma_m}v_{m}\:d\mu_{m} =\int_{\Gamma} \mathrm{Tr}_{D,\Gamma} v\:d\mu\]
	and
	\[\lim_{m\to\infty}\int_{\Gamma_{m}} |\mathrm{Tr}_{D,\Gamma_m}v_{m}|^p\:d\mu_{m} =\int_{\Gamma} |\mathrm{Tr}_{D,\Gamma} v|^p \:d\mu.\]
\end{lemma}
\begin{remark}
	For the definition of the spaces $H^{\beta,p}(D)$ we take the same construction as for $H^{1,2}(\Omega)$ introduced in Section~\ref{SSMeasureBF}.  Given $f\in H^{\beta,p}(D)$, we set 
$\mathrm{Tr}_{D,\Gamma}f:=\mathrm{Tr}_{\Gamma}g,$
where $g$ is an element of $H^{\beta,p}(\mathbb{R}^n)$ such that $f=g|_D$ in $\mathcal{D}'(D)$. The notation $\mathrm{Tr}_{\Gamma}$ only means taking the trace operator on $\Gamma$.
\end{remark}
For the convergence of traces, normal derivatives, the Poincaré-Steklov and layer potentials operators in the image of the trace framework with moving spaces see~\cite{ARXIV-CLARET-2024-1,kuwae-2003}.

Let us notice that the weak convergence of the boundary measures is crucial for converging boundary integrals. Let us fix the perimeter measures of the Lipschitz domains converging to another Lipschitz domain (in all convenient senses, as the Hausdorff, characteristic functions, and compact sets~\cite{HENROT-e} and Section~\ref{Append}). The weak limit of the perimeter measures is not necessarily the perimeter measure of the limit Lipschitz domain. This fact is called non-continuity of the perimeters~\cite {HENROT-e}. It implies that the shape optimization problems in the class of Lipschitz domains corresponding to the minimization of the Robin-type energies~\cite{ARXIV-CLARET-2023,HINZ-2021,HINZ-2021-1,MAGOULES-2021} can only ensure the infimum of these energies. The existence of an optimal shape realizing the energy minimum is proved in the class of uniform (or $(\eps,\infty)$-domains) domains, allowing fractal boundaries viewed as the supports of $d$-upper regular measures. For more details see~\cite{HINZ-2021-1,HINZ-2023}. The most important arguments for the existence of at least one optimal shape are the uniform (on all choosen class of domains) constants in the Poincaré-type inequalities; in the stability estimate of the solution and the uniform boundness of the extension operators $E_{\Omega_k}: H^1(\Omega_k)\to H^1(\R^n)$ pour tout $\Omega_k$ in the fixed (compact) class of domains. 

If the energy convergence is usually directly related to Mosco convergence (see Definition~\ref{DefMC}), it does not generally imply spectrum convergence. In~\cite[Theorem~6]{HINZ-2023}, we show the sufficient conditions on the domains and boundary measures convergences implying the spectrum convergence (as sets in the Hausdorff sense, see Section~\ref{Append} for the definition), the convergence of the resolvents (by the operator norm),  the eigenvalues and eigenfunctions. Let $D\subset \mathbb{R}^n$ be a bounded Lipschitz domain. We take  a sequence of $(\varepsilon,\infty)$-domains $(\Omega_m)_m$ ($\Omega_m \subset D$, $\eps>0$ is fixed) and a sequence  $(\mu_m)_m$ of nonzero Borel measures, all satisfying \eqref{Eqmu} with the same constant $c_d$ and with $\Gamma_m=\supp\mu_m\subset \overline{\Omega}_m$, respectively. The mentioned convergences take place for the corresponding weakly well-posed elliptic-type boundary value problems if $\lim_m \Omega_m=\Omega$ in the Hausdorff sense and in the sense of characteristic functions and $\lim_m \mu_m=\mu$ weakly. 
\begin{remark}
Let us notice that this kind of convergence result does not assume that the sequence $(\Omega_m)_m$ is Lipschitz or more regular. It could be or could not be. It means that the case of the approximation of an irregular by another irregular boundary, as irregular by a regular, is not excluded. 
\end{remark}

Fig.~\ref{FigLocNum} illustrates the convergence behavior of the eigenfunctions established in~\cite[Theorem~6]{HINZ-2023} by a particular example of localized eigenfunctions in ``almost the same region'' for  four domains, denoted by $\Omega_j$, $j=1,\ldots 4$, of constant volume with a changing one side as the four first
 prefractal generations of Minkowski (or rectangular) fractal~\cite{DEKKERS-2022,EVEN-1999,ROZANOVA-PIERRAT-2012,SAPOVAL-1993}.
In particular, we also see the stability behavior of the localization phenomena in the prefractal parts of this sequence of domains.
\begin{figure}[!ht]
  \centering
  \includegraphics[width=0.24\linewidth]{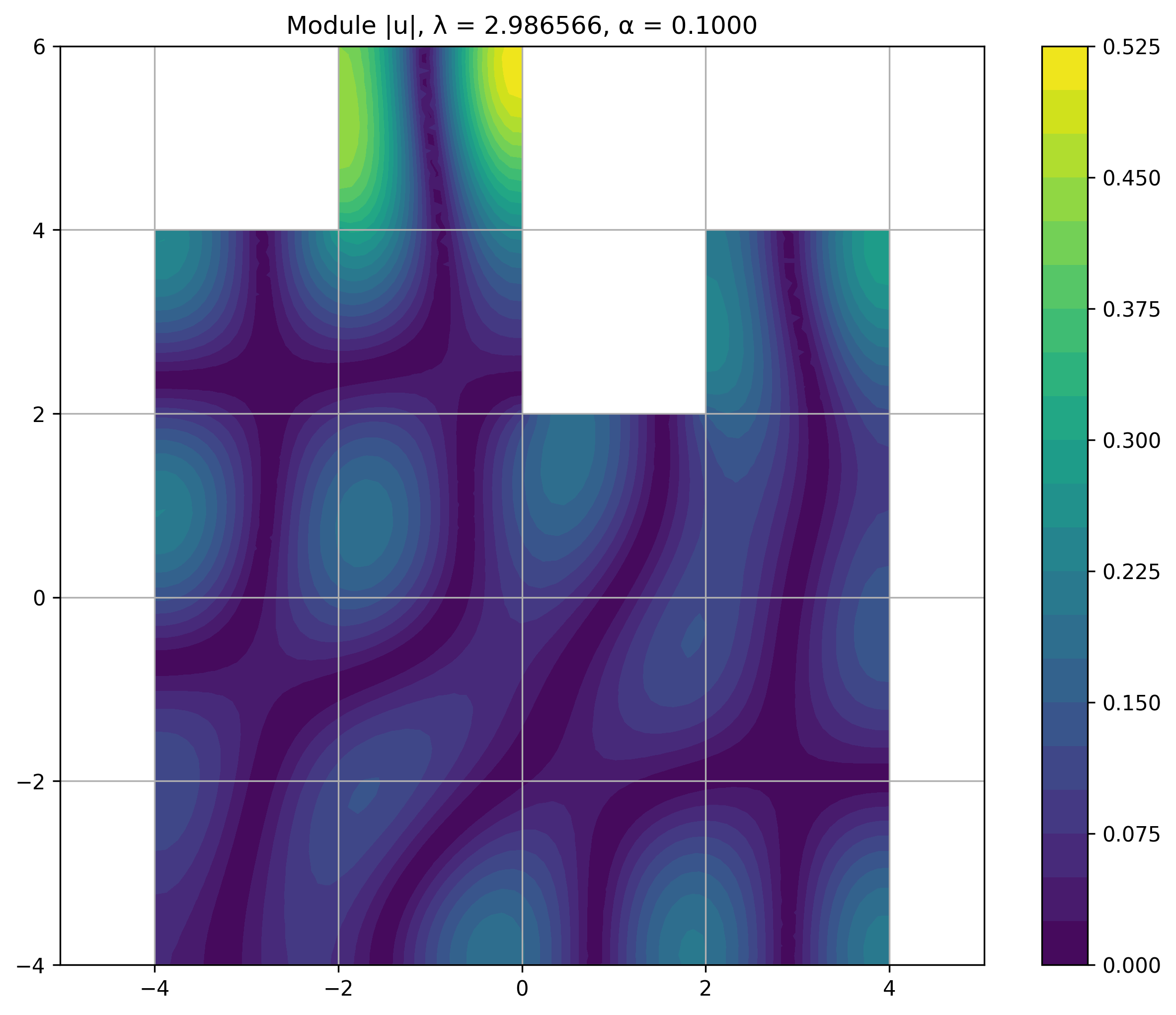}
  \includegraphics[width=0.24\linewidth]{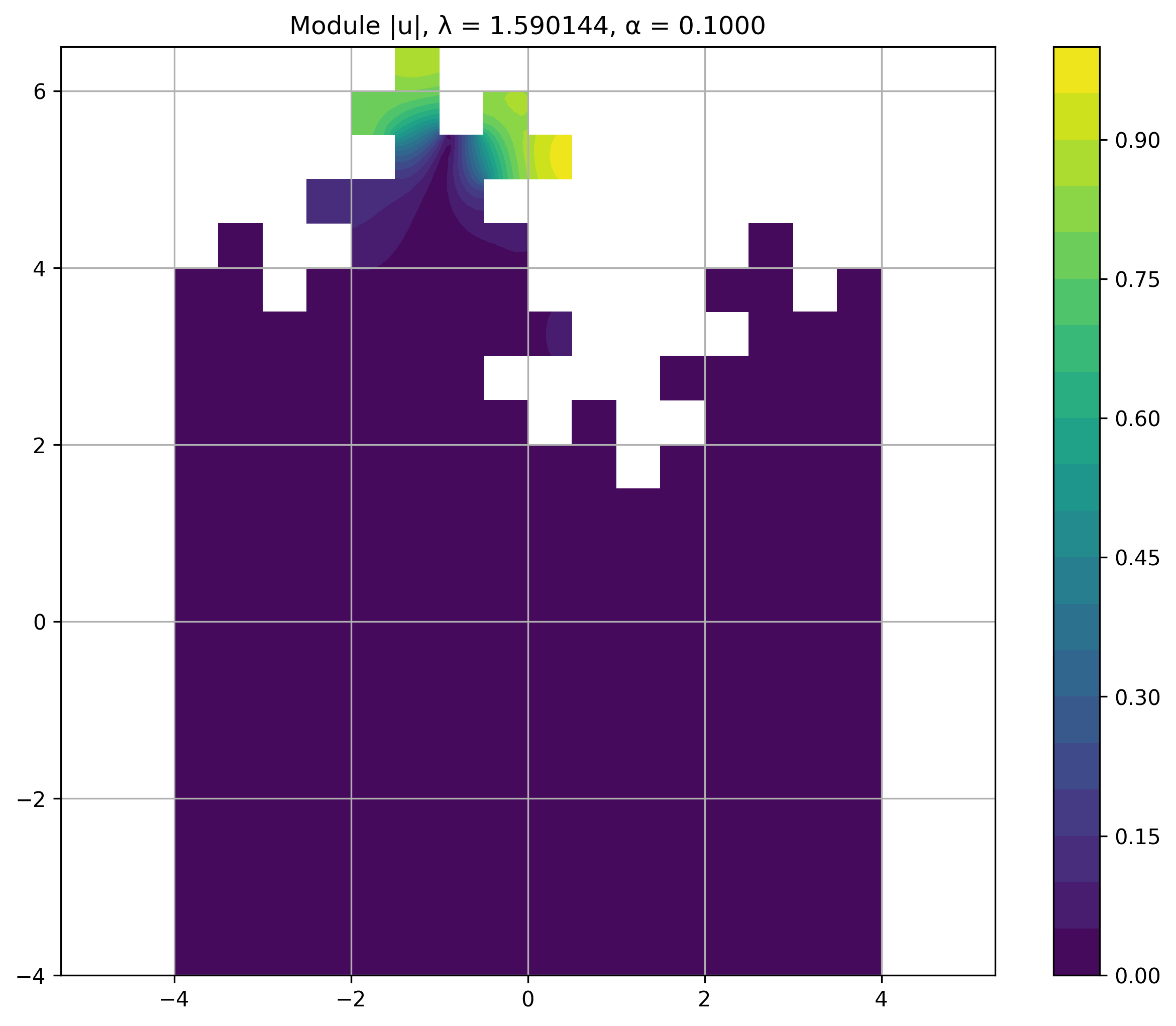}
  \includegraphics[width=0.24\linewidth]{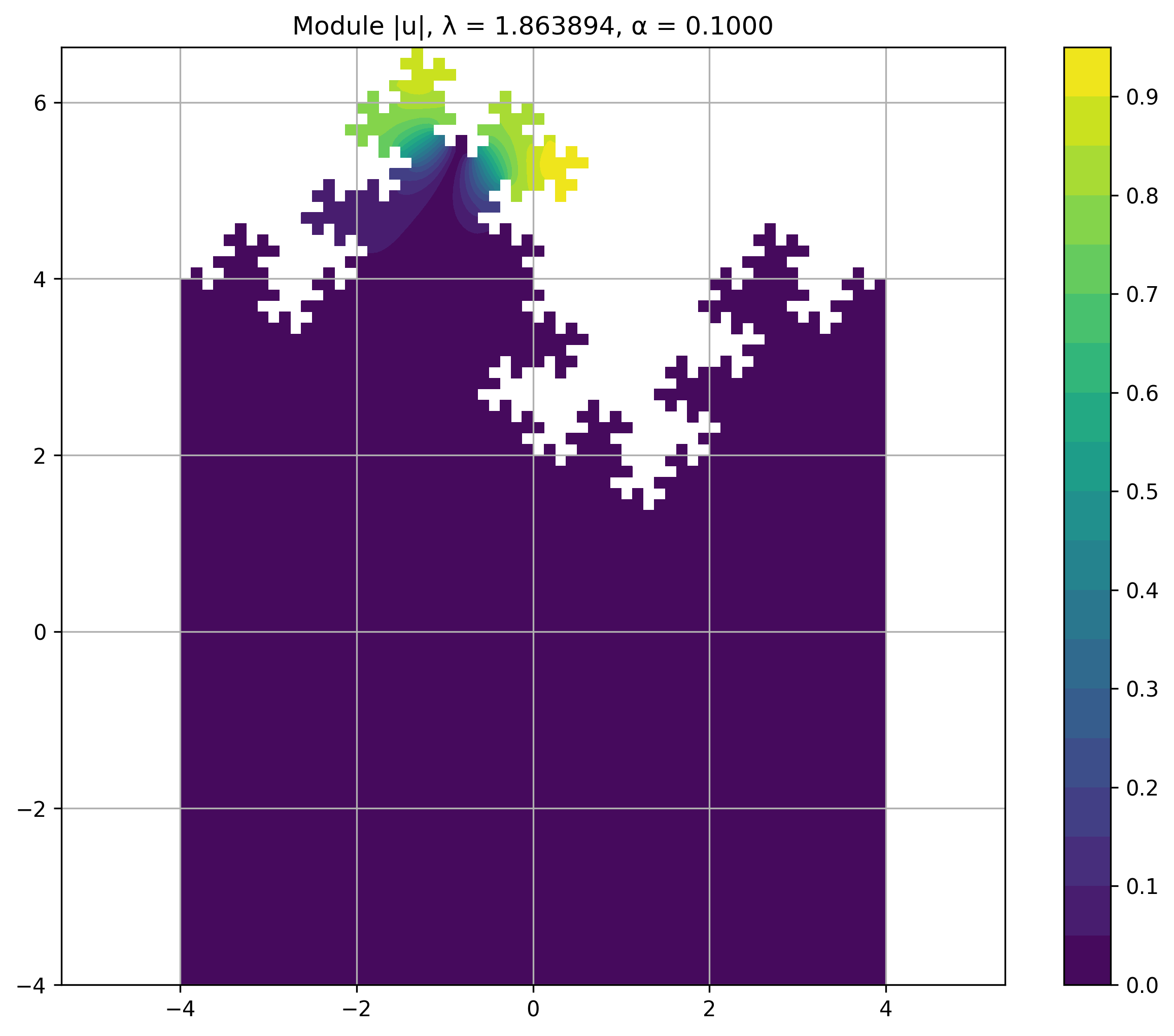}
  \includegraphics[width=0.24\linewidth]{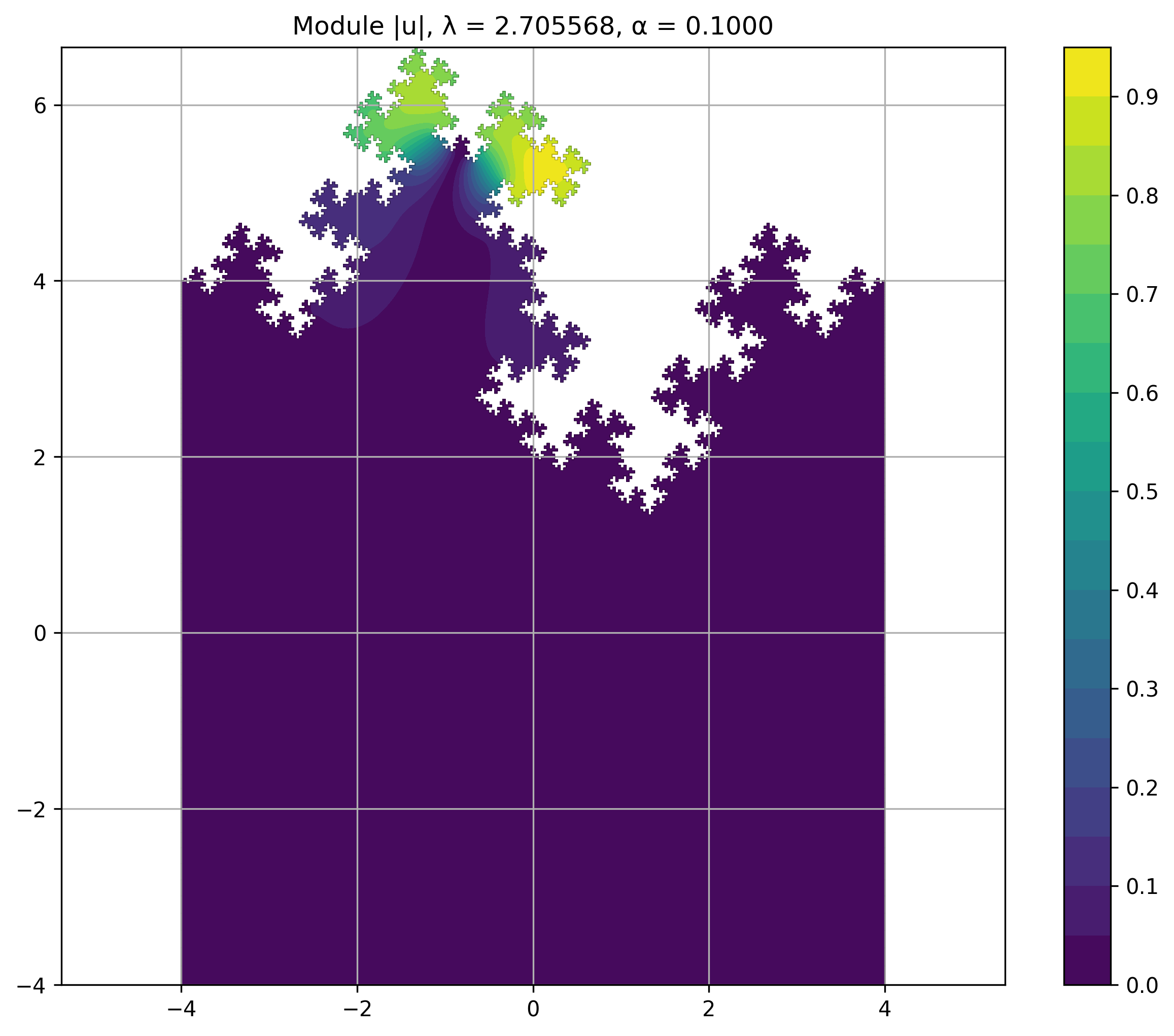}
  \caption{Numerical results by M. Graffin~\cite{GRAFFIN-2025}. We consider the spectral problem for $-\Delta$ with the Robin boundary condition \eqref{EqRC} with $\alpha=0.1$ and perform finite elements using Python. The eigenfunctions $\phi_i^{\Omega_j}$ are plotted in modulus. From the left to right: first prefractal generation $\Omega_1$, $|\phi_{21}^{\Omega_1}|$ associated with the eigenvalue $\lambda \approx 2.99$; second prefractal generation, $|\phi_{14}^{\Omega_2}|$ associated with the eigenvalue $\lambda \approx 1.59$; third prefractal generation, $|\phi_{15}^{\Omega_3}|$ associated with the eigenvalue $\lambda \approx 1.86$; fourth prefractal generation, $|\phi_{19}^{\Omega_4}|$ associated with the eigenvalue $\lambda \approx 2.71$~\cite{GRAFFIN-2025}.}
  \label{FigLocNum}
  \end{figure}


%


\section{Appendix: recall of different definitions}\label{Append}

In this section, we recall several standard definitions used in this article for the reader's convenience.
\begin{definition}[$(\eps,\delta)$-domain~\cite{JONES-1981}]\label{DefUnifD}
An open connected subset $\Omega$ of $\R^n$ is an $(\eps,\delta)$-domain, $\eps > 0$, $0 < \delta \leq \infty$, if whenever $(x, y) \in \Omega^2$ and $|x - y| < \delta$, there is a rectifiable arc $\gamma\subset \Omega$ with length $\ell(\gamma)$ joining $x$ to $y$ and satisfying
\begin{enumerate}
 \item[(i)] $\ell(\gamma)\le \frac{|x-y|}{\eps}$ and
 \item[(ii)] $d(z,\del \Omega)\ge \eps |x-z|\frac{|y-z|}{|x-y|}$ for $z\in \gamma$. 
\end{enumerate}
\end{definition}
The $(\eps,\delta)$-domains give the optimal class of  extension domains~\cite[Theorem~3]{JONES-1981} in $\R^2$, but not in $\R^3$, where there exist extension domains which are not $(\eps,\delta)$-domains. Bounded $(\eps,\infty)$-domains are also called \textit{uniform} domains. All $(\eps,\delta)$-domains are $n$-sets and satisfy the density measure condition, as mentioned in Section~\ref{SubSecspaceB}.
\begin{definition}[Ahlfors $d$-regular set or $d$-set~\cite{JONSSON-1984,JONSSON-1995,WALLIN-1991,TRIEBEL-1997}]\label{defdset}
Let $F$ be a Borel nonempty subset of $\R^n$. The set $F$ is is called a $d$-set ($0<d\le n$) if there exists a $d$-measure $\mu$ on $F$, $i.e.$ a positive Borel measure with support $F$ ($\operatorname{supp} \mu=F$) such that there exist constants 
$c_1$, $c_2>0$,
\begin{equation*}
 c_1r^d\le \mu(F\cap\overline{B_r(x)})\le c_2 r^d, \quad \hbox{ for } ~ \forall~x\in F,\; 0<r\le 1,
 \end{equation*}
where $B_r(x)\subset \R^n$ denotes the Euclidean ball centered at $x$ and of radius~$r$.
\end{definition}
 Thanks to~\cite[Prop.~1, p~30]{JONSSON-1984}, all $d$-measures on a fixed $d$-set $F$ are equivalent.  The restriction of $d$-dimensional Hausdorff measure $m_d$ to $F$, $m_d|_F$, is a $d$-measure on $F$ by~\cite[Thrm~1]{JONSSON-1984}. Thus,
 a $d$-set $F$ has Hausdorff dimension $d$ in the neighborhood of each point of $F$~\cite[p.33]{JONSSON-1984}. Definition~\ref{defdset} includes the case $d=n$, $i.e.$ $n$-sets.

Let us recall now the notion of $H^1$-capacity (\cite[Section 2.3]{CHEN-FUKUSHIMA-2012}, \cite[Section 2.1]{FOT94}  or \cite{ADAMS-1996, MAZ'JA-1985}):
\begin{definition}[Capacity]\label{DefCap}
The \emph{capacity} $\cpct(U)$ of an open set $U\subset\mathbb{R}^n$ is defined by
\[\cpct(U):=\inf\big\lbrace \|u\|_{H^1(\mathbb{R}^n)}^2: u\in H^1(\mathbb{R}^n),\ u\geq 1\ \text{a.e. on $U$}\big\rbrace\]
with the agreement that $\inf \emptyset=+\infty$. The capacity $\cpct(A)$ of a general set $A\subset \mathbb{R}^n$ is defined by
\[\cpct(A):=\inf\big\lbrace \cpct(U):\ A\subset U,\ \text{$U$ open}\big\rbrace.\]
\end{definition}
For $n=1$, all nonempty sets have positive capacity.

Let us also recall that an extended real valued function $v$ defined q.e. on $\mathbb{R}^n$ is \emph{quasi continuous} if for any $\varepsilon>0$ there is an open set $G\subset \mathbb{R}^n$ such that $\cpct(G)<\varepsilon$ and $v$ is continuous on $\mathbb{R}^n\setminus G$.

Recall that the \emph{Hausdorff distance} between two closed sets $K_1, K_2\subset \mathbb{R}^n$ is defined as 
\[d^H(K_1,K_2):=\inf\{\varepsilon>0: K_1\subset (K_2)_\varepsilon \text{ and } K_2\subset (K_1)_\varepsilon\},\]  
where it is used the notation $(K)_\varepsilon=\{x\in\mathbb{R}^n: d(x,K)\leq \varepsilon\}$.
\begin{definition}[Convergence in the Hausdorff sense, Definition~2.2.8~\cite{HENROT-e}]
A sequence $(K_m)_m$ of closed sets $K_m\subset \mathbb{R}^n$ is said to \emph{converge} to a closed set $K\subset \mathbb{R}^n$ \emph{in the Hausdorff sense} if $\lim_{m\to \infty} d^H(K_m,K)=0$.
A sequence $(\Omega_m)_m$ of open sets $\Omega_m\subset D$ is said to \emph{converge} to an open set $\Omega\subset D$ \emph{in the Hausdorff sense} if 
\[d^H(\overline{D}\setminus \Omega_m,\overline{D}\setminus \Omega)\to 0\quad \text{ as }\quad m\to \infty.\] 
\end{definition}
 This definition does not depend on the choice of $D$, \cite[Remark 2.2.11]{HENROT-e}, see also~\cite{shibahara2021gromov}. 

\begin{definition}[Convergence in the sense of characteristic functions, Definition 2.2.3~\cite{HENROT-e}]
A sequence $(\Omega_m)_m$ of open sets $\Omega\subset \mathbb{R}^n$ is said to \emph{converge} to an open $\Omega$ \emph{in the sense of characteristic functions} if 
\[\lim_{m\to \infty} \mathbf{1}_{\Omega_m}=\mathbf{1}_{\Omega}\quad \text{in $L^p_{\loc}(\mathbb{R}^n)$ for all $p\in [1,\infty)$},\]
 in other words, if locally, the Lebesgue measure of the symmetric differences of the domains converges to zero.   
\end{definition}
\begin{definition}[Convergence in the sense of compacts, Definition 2.2.21~\cite{HENROT-e}]
 A sequence $(\Omega_m)$ of open sets $\Omega_m\subset \mathbb{R}^n$ is said to \emph{converge} to an open set $\Omega\subset \mathbb{R}^n$ \emph{in the sense of compacts} if for any compact $K\subset \Omega$ we have $K\subset \Omega_m$ for all sufficiently large $m$ and for any compact $K\subset \mathbb{R}^n\setminus \overline{\Omega}$ we have $K\subset \mathbb{R}^n\setminus \overline{\Omega}_m$ for all sufficiently large $m$. 
 \end{definition}

\begin{definition}[Mosco convergence, Definition 2.1.1~\cite{MOSCO94}]\label{DefMC}
A sequence $(I_m)_m$ of quadratic functionals $I_m:L^2(D)\to [0,+\infty]$ \emph{converges to} a quadratic functional \emph{$I:L^2(D)\to [0,+\infty]$ in the sense of Mosco}, if 
\begin{enumerate}
	\item[(i)] we have $\varliminf_{m\to \infty} I_m(u_m)\ge I(u)$ for every seqence $(u_m)_{m}$ converging weakly to $u$ in $L^2(D)$, 
	\item[(ii)] for every $u\in L^2(D)$ there exists a sequence $(u_m)_{m}$ converging strongly in $L^2(D)$ such that $\varlimsup_{m\to \infty} I_m(u_m)\le I(u)$.
\end{enumerate}
\end{definition}
A sequence $(I_m)_m$ of quadratic functionals $I_m: L^2(D)\to [0,+\infty]$ \emph{converges to} a quadratic functional \emph{$I: L^2(D)\to [0,+\infty]$ in the Gamma-sense} if (ii) above holds and (i) above holds with weak convergence replaced by strong convergence, see \cite[Section 2]{BUCUR-2016}, \cite[Definition 2.2.1]{MOSCO94} or \cite{Braides,DALMASO-1993}. Convergence in the Mosco sense implies convergence in the Gamma sense (see~\cite{HINZ-2023}).

\bibliographystyle{plain} 
\bibliography{biblio.bib}
\end{document}